\documentclass[12pt,leqno]{amsart}
\usepackage{amsmath,amssymb,amsfonts}
\usepackage{eucal,enumerate}
\usepackage{graphicx,psfrag}
\usepackage[utf8]{inputenc}	

\newtheorem{theorem}{Theorem}
\newtheorem{proposition}[theorem]{Proposition}

\newtheorem{lemma}[theorem]{Lemma}

\theoremstyle{definition}

\newtheorem{definition}{Definition}

\renewcommand{\phi}{\varphi}                 
\renewcommand{\epsilon}{\varepsilon}

\newcommand\eset{\varnothing}

\newcommand\cB{\mathcal{B}}
\newcommand\cF{\mathcal{F}}
\newcommand\cN{\mathcal{N}}
\newcommand\bbF{\mathbb{F}}

\begin{document}
\thispagestyle{empty}

\title{On strongly regular signed graphs\\ of higher girth}
\author{Quaid Iqbal}
\address{Department of Mathematics and Statistics,  Binghamton University, Binghamton, NY 13902-6000, U.S.A.}
\email{qiqbal@math.binghamton.edu}
\author{Thomas Zaslavsky}
\address{Department of Mathematics and Statistics,  Binghamton University, Binghamton, NY 13902-6000, U.S.A.}
\email{zaslav@math.binghamton.edu}

\begin{abstract}
Strongly regular signed graphs are an extension of strongly regular graphs to the realm of signed graphs, that is, graphs where each edge is positive or negative.  
Unlike with ordinary strongly regular graphs, most kinds of signed counterparts with girth 4 or higher are describable in terms of known structures.  We prove that those with girth 4 that are bipartite are classified by designs of two kinds: weighing matrix designs and symmetric block designs.  Those of girth 5 are few and readily described.  There are none of higher girth.  Those with girth 4 that are not bipartite are unsolved.
\end{abstract}



\maketitle

\section{Introduction}

Strongly regular graphs are a lively part of graph theory for their applications to design theory, group theory, and geometry, but they are difficult; e.g., there seems no hope of classifying them.  Signed graphs add structure to graphs by putting a positive or negative sign on every edge, with ordinary graphs corresponding to signed graphs with only one sign (usually positive).  The problem presented itself of finding a generalization of strongly regular graphs to signed graphs.  Stani\'c \cite{S} proposed a good definition that has led to an increasing literature.

\begin{definition}\label{srsgdef}
A \emph{signed graph} $\Sigma$ consists of an underlying simple graph $\Gamma = (V,E)$ and a sign function $\sigma: E \to \{+1,-1\}$.  It is \emph{homogeneous} if $\sigma$ is a constant function.  The \emph{positive subgraph} $\Sigma^+$ consists of all vertices and only the positive edges; the \emph{negative subgraph} $\Sigma^-$ is similar.  
The sign of a walk is the product of the signs of its edges (counting multiplicity).  It is \emph{strongly regular} if it is regular (that is, $\Gamma$ is regular), it is neither homogeneous and incomplete nor totally disconnected, and there exist constants $a, b, c$ such that $w_2^\pm(u,v) = a$ for every positive edge $uv$, $b$ for every negative edge $uv$, and $c$ for every nonedge $uv$, where $u,v$ denote distinct vertices and $w_2^\pm(u,v)$ means the number $w_2^+(u,v)$ of positive 2-paths $P_{uv}$ less the number $w_2^-(u,v)$ of negative 2-paths.
\end{definition}

In this article we do something that has not yet been possible for ordinary strongly regular graphs:  we give a complete description of all nontrivial strongly regular signed graphs that are bipartite or that have girth 5.  Nontriviality means we exclude those with constant sign, which are essentially ordinary strongly regular graphs.  The main part of the classification is for bipartite graphs of girth 4, which shows they are fully described by weighing matrices (if $c = 0$) and symmetric balanced incomplete block designs (if $c \neq 0$).  
Most of this was proved by Stani\'c in \cite[Section 5]{S} using eigenvalues, but also assuming that both $\Sigma^+$ and $\Sigma^-$ are regular (which he stated, equivalently, as net regularity when $c \neq 0$).  
Our bipartite proof is purely combinatorial and does not assume net regularity; thus our proof shows that all examples not from weighing matrices are the ones found by Stani\'c.  Our proof for girth 5 is also purely combinatorial.  Higher girth than that is impossible.

The definition of a strongly regular signed graph (SRSG) can be restated in terms of adjacency matrices.  Letting $A(\Gamma) = (a_{uv})_{u,v\in V}$ be that of the underlying graph, that of the signed graph is $A(\Sigma) = (\sigma(uv)a_{uv})_{u,v\in V}$.  Then the condition for strong regularity is that 
\begin{equation*}
A(\Sigma)^2 = \frac{a}{2}(A(\Sigma) + A(\Gamma)) - \frac{b}{2}(A(\Sigma) - A(\Gamma)) + cA(\bar\Gamma) + rI,
\label{srsg-matrix}
\end{equation*}
where $\bar\Gamma$ is the complementary graph and $r$ is the degree of $\Sigma$.  When the graph has girth greater than 3, as in this paper, this equation becomes
\begin{equation}
A(\Sigma)^2 = cA(\bar\Gamma) + rI.
\label{bsrsg-matrix}
\end{equation}

An important operation on a signed graph is switching.  \emph{Switching} a vertex set means reversing the signs of all edges between a vertex subset $X$ and its complement.  In terms of the adjacency matrix, this replaces $A(\Sigma)$ by $A(\Sigma^X) := S_{X} A(\Sigma) S_{X}$ where $S_{X}$ is the $V \times V$ diagonal matrix with $s_{vv} = -1$ if $v \in X$ and $+1$ if $v \notin X$.  A switching is \emph{trivial} if it does not change edge signs. 
Switching preserves the adjacency spectrum but, usually, not strong regularity.

We do not find the strongly regular signed graphs that have girth 4 but are not bipartite.  That seems to be difficult.  We conjecture that they are homogeneously signed; we can prove that for the small Clebsch graph.

\section{Girth $4$, bipartite}

We consider a strongly regular signed graph that is bipartite, with color classes $U = \{u_1,\dots,u_n\}$ and $W = \{w_1,\dots,w_n\}$.  (Regularity implies that $|U|=|W|$.)  The adjacency matrix of such a signed graph has the form $A = \left(\begin{smallmatrix} O & B \\ B^T &O \end{smallmatrix}\right)$, where $B$ is an $n \times n$ matrix.  We call $B$ the \emph{bipartite adjacency matrix}.  The matrix $J$ is the $n \times n$ all-1's matrix.

The positive, or negative, neighborhood of a vertex $v$ is the set of vertices that are positively, or negatively, adjacent to it and is denoted by $N^+(v)$, or $N^-(v)$, respectively.

The significant bipartite signed graphs we encounter are of three kinds.  

\begin{enumerate}

\item \emph{Bipartite weighing signed graphs}, whose bipartite adjacency matrix $B$ is a weighing matrix.  A \emph{weighing matrix} is a square matrix $W$ whose elements are in the set $\{0,\pm1\}$, such that $WW^T = rI$ for some integer $r$ (which is the number of nonzeroes in each row and column).  The underlying graph in this type need not be complete bipartite.

\item[(1a)] \emph{Bipartite Hadamard signed graphs}, whose bipartite adjacency matrix $B$ is a Hadamard matrix.  This is the special case of a weighing matrix where the underlying graph is complete bipartite (that is, $r=n$).

\item \emph{Symmetric-design signed graphs}, for which $B = J-2S$ where $S$ is the incidence matrix of a symmetric 2-design.  More explicitly, $|\Sigma| = K_{n,n}$ and the sets $N^-(u_i)$, $i = 1,\dots, n$ [and equivalently, the sets $N^-(v_i)$, $i = 1,\dots, n$] are the blocks of a symmetric design on $n$ elements.  

\item \emph{Perfect-matching signed graphs}, for which $B = J - 2M$ where $M$ is the incidence matrix of a perfect matching (and by relabelling, $B = J - 2I$).  We regard these as trivial symmetric-design signed graphs, where the block size is 1 and $\lambda=0$.

\end{enumerate}

We remark that a Hadamard matrix of order $n$ gives rise to a symmetric design on $n-1$ elements; this implies that a bipartite Hadamard signed graph gives rise to a symmetric-design signed graph by (after appropriate switching) deleting a vertex incident to only positive edges in each of $U$ and $W$.  Conversely, a Hadamard-design signed graph on $2(n-1)$ vertices gives a bipartite Hadamard signed graph by adding one vertex to each of $U$ and $W$ that are each positively adjacent to every vertex in the opposite set.  However, the bipartite Hadamard signed graph is not itself a symmetric design signed graph.

\begin{theorem}\label{structure4}
A bipartite signed graph $\Sigma$ with girth $4$ and $|U|, |W| \geq 2$ is strongly regular if and only if it is a bipartite weighing signed graph, a symmetric-design signed graph, a perfect-matching signed graph, or a bipartite strongly regular graph of girth $4$ with homogeneous signs.
\end{theorem}

\begin{lemma}[Switching Lemma]\label{sw}
Nontrivial switching destroys strong regularity if $c \neq 0$ and preserves it (and the value of $c$) if $c=0$.
\end{lemma}

\begin{proof}
We apply switching to Equation \eqref{bsrsg-matrix}.  We may assume the graph is connected.  The switching matrix of $A = \left(\begin{smallmatrix} O & B \\ B^T &O \end{smallmatrix}\right)$ has the form $S = \left(\begin{smallmatrix} S_{X \cap U} &O \\ O &S_{X \cap W} \end{smallmatrix}\right)$, so 
$$
S A S = \left(\begin{matrix} O &S_{X \cap U} B S_{X \cap W} \\[6pt] S_{X \cap W} B^T S_{X \cap U} &O \end{matrix}\right).
$$
The switched Equation \eqref{srsg-matrix} is
$$
A(\Sigma^X)^2 = c \begin{pmatrix} S_{X \cap U} J S_{X \cap U} & O \\[6pt] O & S_{X \cap W} J S_{X \cap W} \end{pmatrix} + rI,
$$
since $\bar\Gamma$ is the complement of $K_{n,n}$.  For the switched graph $\Sigma^X$ to be strongly regular, it must still satisfy Equation \eqref{srsg-matrix}; thus,  
$$
A(\Sigma^X)^2 = c' \begin{pmatrix} J & O \\ O & J \end{pmatrix} + rI,
$$
where $c'$ is a constant.  Combining these equations,
$$
c \begin{pmatrix} S_{X \cap U} J S_{X \cap U} & O \\[6pt] O & S_{X \cap W} J S_{X \cap W} \end{pmatrix} = c' \begin{pmatrix} J & O \\[6pt] O & J \end{pmatrix}.
$$
There are two possible solutions.  One is that $c'=c=0$.  The other is that $S_{X \cap U} J S_{X \cap U} = J = S_{X \cap W} J S_{X \cap W}$.  The only solutions to the latter are that $X \cap U = \eset$ or $U$, and $X \cap W = \eset$ or $W$, which implies that no edge signs were changed.  Thus, if $c \neq 0$, nontrivial switching results in a signed graph that is not strongly regular.

Since $c=0$ implies $c'=c=0$, the value of $c$ is preserved.
\end{proof}

\begin{proof}[Proof of Theorem \ref{structure4}]
We assume $n>2$ in the first two cases and deal with tiny $n$ in Case 3.

\textbf{Case 1.}  $c=0$.  Then the bipartite adjacency matrix $B$ is a weighing matrix because $A^2 = \left(\begin{smallmatrix} BB^T & O \\ O & B^TB \end{smallmatrix}\right) = rI$, so $BB^T = B^TB = rI$, which for a matrix with entries $0, \pm1$ is the definition of a weighing matrix.

Conversely, if $B$ is a weighing matrix, then $A^2 = \left(\begin{smallmatrix} BB^T & O \\ O & B^TB \end{smallmatrix}\right) = rI$, which satisfies Equation \eqref{bsrsg-matrix}.

\textbf{Case 2.}  $c \neq 0$.

We first prove that $G = K_{n,n}$.  Because $w_2^\pm = c \neq 0$, no two vertices are at distance greater than 2.  Therefore $G$ is complete bipartite and $r=n$.

Now we study the negative neighborhoods of vertices in $\Sigma$.  We write $U = \{u_1,u_2,\dots,u_n\}$ and $W = \{w_1,w_2,\dots,w_n\}$ for the two color classes of the $K_{n,n}$.  For any two vertices $u_i, u_j \in U$, we have $w_2^+(u_i, u_j) + w_2^-(u_i, u_j) = n$ and $w_2^+(u_i, u_j) - w_2^-(u_i, u_j) = c$, so $w_2^+(u_i, u_j) = \frac12(n+c)$ and $w_2^-(u_i, u_j) = \frac12(n-c)$, independent of the choice of vertices.  
Observe that $w_2^-(u_i, u_j) = | N^-(u_i) \oplus N^-(u_j) |$; thus, the symmetric difference of any two negative neighborhoods of vertices in $U$ has constant size.

We pause to verify that a symmetric-design signed graph is strongly regular.  Let $\Sigma_D$ be such a graph (with the same vertices as $\Sigma$), whose design $D$ has parameters $(n,k,\lambda)$. 
Then $k = | N^-(u_i) |$ for all $i$ and $\lambda = | N^-(u_i) \cap N^-(u_j) |$ for $i \neq j$.  Thus, $| N^-(u_i) \oplus N^-(u_j) | = w_2^-(u_i,u_j) = 2(k-\lambda)$ and $c_D = w_2^\pm(u_i,u_j) = n - 4(k-\lambda)$.  This establishes that $\Sigma_D$ is strongly regular.  

Now we apply a fact from coding theory.  A family of vectors in $\bbF_2^n$ is called \emph{equidistant} if for every pair in the set, the sum has the same number of nonzero components.  A set-theoretic interpretation is that a family $\cN$ of subsets of a set $X$ is \emph{equidistant} if there is a number $d$ such that every pair of sets in $\cN$ has symmetric difference of size $d$.  There is a remarkable theorem.  For $X \subseteq U$, define $\cF \oplus X := \{ A \oplus X : A \in \cF \}$; we call this operation on $\cF$ \emph{set switching}.

\begin{proposition}[{See \cite[Theorem 4.7.5]{CSD}}] \label{is}
Let $X$ be a set of size $n\geq3$ and let $\cN$ be a family of $n$ subsets of $X$ such that $|A \oplus B| = d$, a constant such that $d \neq n/2$, for every pair of sets $A, B \in \cB$.  Then $\cN$ set-switches to the family of blocks of a symmetric block design.
\end{proposition}

We apply this proposition to the sets $N_i = N^-(u_i) \subseteq W$.  Here $w_2^-(u_i,u_j) = | N_i \oplus N_j | = \frac12(n-c)$ for all $i \neq j$, so the proposition applies with $d = \frac12(n-c) \neq n/2$.  Set switching by $X \subseteq U$ is equivalent to sign switching $U$ in $\Sigma$.  Therefore, Proposition \ref{is} tells us that the sets $N_i$ are the result of set-switching a symmetric design on $U$, thus $\Sigma$ is the result of switching a symmetric-design signed graph.  

By Lemma \ref{sw}, if strongly regular $\Sigma$ is a switching of a symmetric-design graph $\Sigma_D$ with $c_D \neq 0$, then $\Sigma = \Sigma_D$ if $c_D \neq 0$.  But $c_D = 0$ is impossible by the same lemma, since then $c$ would be 0, which is Case 1.  It follows that $\Sigma$ itself is a symmetric-design signed graph.

\textbf{Case 3.}  Suppose $n\leq 2$; then girth 4 implies $n=2$ and the underlying graph is a quadrilateral.  The only possibility that is not homogeneously signed is that adjacent edges have opposite signs.
\end{proof}

We conclude this section with a few remarks about inhomogeneous $\Sigma$ with $c \neq 0$.  Such a $\Sigma$ is constructed from a symmetric 2-design with parameters $(n,k,\lambda)$.
\begin{enumerate}

\item The underlying graph is regular.  The positive and negative subgraphs are also regular of degrees $n-k$ and $k$.

\item If $\Sigma^-$ is not a matching (i.e., the negative degree is $>1$), then $\Sigma^-$ is connected.  

If it were not connected, then $c = w_2^\pm$ would not be well defined, because $w_2^-(u_i,u_j) = |N_i \oplus N_j| = 2k$ for vertices in different components of $\Sigma^-$, since $N_i \cap N_j = \eset$, but for vertices in the same component, $w_2^-(u_i,u_j) = 2(k-\lambda)$, since $|N_i \cap N_j| = \lambda > 0$.  For the same reason $\Sigma^+$ is connected (e.g., by negating the edges of $\Sigma$) unless it is a matching.

\item Negative degree $k=2$ (and by reversing signs, positive degree 2) is impossible because no such symmetric design exists.  The design must satisfy $\lambda(n-1) = k(k-1) = 2$, but then $n\leq3$, which does not admit an inhomogeneous example.

\item Negative degree $k=3$ implies $\lambda(n-1) = 6$.  The design is the Fano plane if $\lambda =1$.  If $\lambda =2$, it is the complement of a perfect matching; i.e., the positive subgraph is a perfect matching.  Any larger $\lambda$ implies $n \leq 3$.

\end{enumerate}

\section{Higher girth}

\begin{proposition}\label{structure6}
A connected strongly regular signed graph must have diameter at most $2$.  It has girth at most $5$, or it is $K_{n,1}$.
\end{proposition}

\begin{proof}
Suppose $\Sigma$ has two vertices $u,v$ at distance 3.  Let $uxyv$ be a path of length 3 joining those vertices.  Then $w^\pm_2(u,y) = \pm1$ but since there is no path of length 2 joining $u$ and $v$, $w^\pm_2(u,v) = 0$.  This contradicts the definition of a strongly regular signed graph, so girth $g \geq 6$ is impossible.

Suppose $\Sigma$ has girth $g \ge 6$, possibly infinite.  Then it has two vertices $u,v$ at distance 3, which is impossible, or it is a tree with diameter 2, which is $K_{n,1}$.
\end{proof}

We dispose of infinite girth:

\begin{proposition}
A strongly regular signed graph with underlying graph $K_{n,1}$ must be homogeneous, except when $n=2$.
\end{proposition}

\begin{proof}
Homogeneity is needed to make $w_2^\pm$ well defined, if $n>2$.
\end{proof}

We show that a strongly regular signed graph of girth 5 is trivial (that is, from the viewpoint of signed graphs; as they are the Moore graphs, they are not entirely known).

\begin{theorem}\label{structure5}
A connected signed graph $\Sigma$ with girth $5$ is strongly regular if and only if it is a homogeneously signed strongly regular graph of girth $5$.
\end{theorem}

The proof follows three lemmas.

\begin{lemma}\label{srsg5}
If $\Sigma$ is strongly regular with girth $5$, then its underlying graph is strongly regular with $\lambda=0$ and $\mu=1$.
\end{lemma}

\begin{proof}
Suppose $u, v$ are nonadjacent vertices.  By Theorem \ref{structure6}, they have a common neighbor.  Girth 5 implies they cannot have a second common neighbor, so $\mu = 1$.
\end{proof}

\begin{lemma}\label{homoC5}
If $\Sigma$ is strongly regular with girth $5$, then every $C_5$ in $\Sigma$ is homogeneous.
\end{lemma}

\begin{proof}
Girth 5 implies that $\mu=1$, so there is a unique 2-path between any two nonadjacent vertices $x, y$.  Hence, by strong regularity $w_2(x,y) = \pm1 = \alpha$, the same for every such pair $x,y$.  
Consider a $C_5$ with edges $e_1e_2e_3e_4e_5$ in cyclic order.  We compute 2-path signs in the circle: 
$\sigma(e_1e_2) = \sigma(e_3e_4) = \alpha$, so
$\sigma(C_5) = \sigma(e_1e_2)\sigma(e_3e_4)\sigma(e_5) = \sigma(e_5).$  This establishes that every edge in the $C_5$ has the same sign as the pentagon sign, so the $C_5$ is homogeneous.
\end{proof}

It follows that if two pentagons have a common edge, then every edge of both has the same sign.

\begin{lemma}\label{srg5}
In a strongly regular graph with girth $5$, every pair of edges belongs to a common pentagon.
\end{lemma}

\begin{proof}
If the edges are adjacent, say they are $xy$ and $yz$.  Let $u \sim x$; then $uxyz$ is a path of length 3, which means there is either an edge $uz$ or a 2-path $uvz$ since the diameter is 2.  The former case gives a quadrilateral, so we must have the 2-path, hence a pentagon $uxyzvu$ that contains both edges.

If the edges are nonadjacent, say they are $xy$ and $zw$.  If there is an edge between them, such as $yz$, then there is no other edge between them, due to the girth assumption.  Diameter 2 implies a 2-path $xuz$, which forms a pentagon $xyzwux$ that contains the edges $xy$ and $zw$.
\end{proof}

\begin{proof}[Proof of Theorem \ref{structure5}]
By Lemmas \ref{srsg5} and \ref{srg5}, every pair of edges in $\Sigma$ lies in a pentagon.  By Lemma \ref{homoC5} the edges have the same sign.  That completes the proof of Theorem \ref{structure5}.
\end{proof}


\end{document}